\documentclass[12pt]{amsart}

\usepackage{amssymb}
\usepackage{amsmath}
\usepackage{amsthm}
\usepackage{mathtools}
\usepackage{multirow}
\usepackage{thmtools}
\usepackage{tikz}
\usetikzlibrary{calc}
\usetikzlibrary{braids}
\usepackage[colorlinks, citecolor=blue]{hyperref}
\theoremstyle{plain}

\newtheorem{theorem}{Theorem}[section]
\newtheorem{corollary}[theorem]{Corollary}
\newtheorem{problem}{Problem}[section]
\usepackage{geometry}
\begin{document}
	
\title[On a family of hyperbolic Brunnian links and their volumes]{On a family of hyperbolic Brunnian \\ links  and their volumes}

\author[D.\,D.~Repov\v{s}]{Du\v{s}an~D.~Repov\v{s}} 
\address{Faculty of Education, Faculty of Mathematics ad Physics, University of Ljubljana
\break
\&
Institute of Mathematics, Physics and Mechanics, 
Ljubljana, 1000, Slovenia
\break
 \url{https://orcid.org/0000-0002-6643-1271}}
\email{\href{mailto:dusan.repovs@guest.arnes.si}{dusan.repovs@guest.arnes.si}}
\noindent
\author[A.\,Yu.~Vesnin]{Andrei~Yu.~Vesnin} 
\address{Sobolev Institute of Mathematics, Russian Academy of Sciences, Novosibirsk, 630090
\break
 \& Regional Mathematical Center, Tomsk State University, Tomsk, 634050, Russia
 \break
 \url{https://orcid.org/0000-0001-7553-1269}}
\email{\href{mailto:vesnin@math.nsc.ru}{vesnin@math.nsc.ru}}
\thanks{D.R. was supported by the Slovenian Research and Innovation Agency (program P1-0292 and grants J1-4031, J1-4001, N1-0278, N1-0114, and N1-0083).
A.V. was supported by the Ministry of Science and Higher Education of Russia under agreement no. 075-02-2024-1437.} 
\date{}
	
\begin{abstract}  
An $n$-component link $L$ is said to be \emph{Brunnian} if it is non-trivial but every proper sublink of $L$ is trivial. The simplest and best known example of a hyperbolic Brunnian link is the 3-component link known as "Borromean rings". For $n\geq 2,$ we introduce an infinite family of $n$-component Brunnian links with positive integer parameters $Br(k_1, \ldots, k_n)$  that generalize examples constructed by Debrunner in 1964. We are interested in hyperbolic invariants of 3-manifolds $S^3 \setminus Br(k_1, \ldots, k_n)$ and
 we
obtain  upper bounds for their volumes. Our approach is based on Dehn fillings on cusped manifolds with volumes related to volumes of ideal right-angled hyperbolic antiprisms.
 \end{abstract} 
	
\subjclass[2020]{57K10, 57K32, 52B10}	
\keywords{Hyperbolic Brunnian link, Adams move, augmented link, ideal right-angled antiprism} 
	
\maketitle

\section{Introduction}

If a link $L$ in $S^3$ is nontrivial, yet every proper sublink of $L$ is trivial, we say that $L$  is \emph{Brunnian}\index{Brunnian link} (or that $L$ has the \emph{Brunnian property}). Links with this property were studied by Milnor~\cite{Mi54}
who
 called them \emph{almost trivial} links. In 1961 Debrunner~\cite{De61} renamed them 
 as
  \emph{Brunnian} links, in honor of Hermann Brunn whose early contributions~\cite{Br92} to knot theory also included
   examples of such links. In recent decades, Brunnian links have been under investigation from several points of views. In particular, we refer to Bai~\cite{Ba21}, 
   Bai and Ma~\cite{BM21},
    Bai and Wang~\cite{BW20},
     and Kanenobu~\cite{Ka84, Ka86} for hyperbolic and satellite properties of Brunnian links, 
     Lei, Wu and Zhang~\cite{LWZ} for intersecting subgroups of Brunnian link groups, and 
     Habiro and Meilhan~\cite{HM08}
      for finite-type invariants  and Milnor invariants of Brunnian links. 
      We would also like  to mention molecular Borromean rings considered by 
      Wang and Stoddart~\cite{WS17}
       and a  possible physical realization of higher-order Brunnian structures studied 
       by Baas~\cite{Ba13}.

The simplest and best known example of a 3-component Brunnian link is the 3-component link $6^3_2$ in Rolfsen's notations~\cite{Ro76} which is also known as the ''Borromean rings''\index{Borromean rings}. We denote it by $\mathcal B$. It is well-known from Thurston~\cite[Chapter~3]{Th80}
 that link $\mathcal B$ is hyperbolic, the complement $S^3 \setminus \mathcal B$ can be decomposed into two copies of an ideal right-angled octahedron, and $\operatorname{vol} (S^3 \setminus \mathcal B) = 7.327724,$ up to six digits.  

We recall that a polyhedron in a hyperbolic 3-space $\mathbb H^3$ is said to be \emph{ideal}\index{polyhedron|ideal} if all of its vertices belong to $\partial \mathbb H^3$, and
it
 is called  \emph{right-angled}\index{polyhedron|right-angled} if all of its dihedral angles are   equal to $\pi/2$. 
 By
  virtue of 
  the Andreev theorem~\cite{An70}, 
  every ideal right-angled hyperbolic polyhedron is determined by its 1-dimensional skeleton, up to an isometry of $\mathbb H^3$. An initial list of ideal right-angled hyperbolic polyhedra and their volumes was  presented by 
  Egorov and Vesnin~\cite{EV20-1, EV20-2}, 
  and  the upper volume bounds depending only of number of vertices  were obtained by 
  Alexandrov, Bogachev, Vesnin, and Egorov~\cite{ABVE}. 

In the present chapter, we introduce for every $n\geq 2,$ an infinite family of $n$-component Brunnian links 
which
 generalize examples constructed  in 1961
 by 
 Debrunner~\cite{De61}.  
 We are interested in the hyperbolic structure on the complements of these links. By using decompositions of complements of fully augmented links  into pairs of ideal right-angled polyhedra by a method from an appendix by 
 Agol and Thurston in
 Lackenby~\cite{La04}, 
 and the
 Dehn filling theorem, we provide upper bounds for volumes of such hyperbolic Brunnian links.

The chapter is organized as follows: In Section~\ref{sec2}  we 
 introduce a family of $(3n+2)$-component  links $L_n,$ for every $n \geq 2,$ and demonstrate that their complement $S^3 \setminus L_n$ can be decomposed into four ideal right-angled antiprisms $A_{2n}$. Then we 
 apply Thurston's formula 
from~\cite[Chapter 6]{Th80} 
for volumes of antiprisms to obtain a formula for $\operatorname{vol} (S^3 \setminus L_n)$ (see 
Theorem~\ref{th1}).  Next, applying Adams moves to $2n$ vertical components of $L_n$, we 
construct links $L'_n$ with $3n$ components such that $\operatorname{vol} (S^3 \setminus L'_n) = \operatorname{vol} (S^3 \setminus L_n)$. In Section~\ref{sec3}, we 
construct by Dehn fillings
 on $2n$ components of $L'_n,$ a family of $n$-component links $Br(k_1, \ldots, k_n),$ depending on
the
 filling parameters $k_1, \ldots, k_n$  (see Theorems~\ref{th3} and \ref{th5}). In particular, when all $k_i$ are equal to $1$, we get the Brunnian links from 
Debrunner~\cite{De61}, whose hyperbolicity was established by 
Bai~\cite{Ba21}.     

\section{Hyperbolic links and ideal right-angled antiprisms} \label{sec2}

Recall that a knot or a link $K \subset S^3$ is said to be \emph{hyperbolic}\index{hyperbolic link} if its complement $S^3 \setminus K$ admits a complete metric of constant curvature -1. Equivalently, the 3-manifold $S^3 \setminus K = \mathbb H^3 / G$ is hyperbolic,  where $\mathbb H^3$ is 
the
hyperbolic 3-space and $G$ is a discrete, torsion-free group of isometries, isomorphic to $\pi_1 (S^3 \setminus K)$. 

For $n\geq 2,$ let us denote a link with $3n+2$ components by $L_n,$  where each component is a circle and the components are linked in the same manner as shown in Figure~\ref{fig1} for the case $n=3$. 

\begin{figure}[h]
\begin{center} 
\scalebox{0.8}{
\begin{tikzpicture} 
%\draw[step=5.mm, gray, very thin] (0,0) grid (100.mm,50.mm);
\draw  [ultra thick, black]  (0,0) -- (7.5,0);
\draw  [ultra thick, black]  (0,5) -- (7.5,5);
\draw[ultra thick, black] (0,1) arc (90:270:0.5);
\draw[ultra thick, black] (0,5) arc (90:270:0.5);
\draw[ultra thick, black] (7.5,1) arc (90:-90:0.5);
\draw[ultra thick, black] (7.5,5) arc (90:-90:0.5);
 %%%   
\draw[ultra thick, black] (0,3) arc (90:270:0.5);
\draw  [ultra thick, black]  (0,2) -- (0.5,2);
\draw  [ultra thick, black]  (0,3) -- (0.5,3);
\draw  [ultra thick, black]  (0.8,2) -- (1.5,2);
\draw  [ultra thick, black]  (0.8,3) -- (1.5,3);
\draw[ultra thick, black] (1.5,3) arc (90:-90:0.5);
\draw  [ultra thick, black]  (0,1) -- (0.5,1);
\draw  [ultra thick, black]  (0,4) -- (0.5,4);
\draw  [ultra thick, black]  (0.8,4) -- (3.5,4);
\draw  [ultra thick, black]  (0.8,1) -- (3.5,1);
\draw[ultra thick, red]  (1,2.1)  arc (50:320: 0.2cm and 0.9cm);
\draw  [ultra thick, red]  (1,1.2) -- (1,1.8);
\draw[ultra thick, red]  (1,4.1)  arc (50:320: 0.2cm and 0.9cm);
\draw  [ultra thick, red]  (1,3.2) -- (1,3.8);
%%%
\draw[ultra thick, red]  (4,2.1)  arc (50:320: 0.2cm and 0.9cm);
\draw  [ultra thick, red]  (4,1.2) -- (4,1.8);
\draw[ultra thick, red]  (4,4.1)  arc (50:320: 0.2cm and 0.9cm);
\draw  [ultra thick, red]  (4,3.2) -- (4,3.8);
\draw[ultra thick, black] (3,3) arc (90:270:0.5);
\draw  [ultra thick, black]  (3,2) -- (3.5,2);
\draw  [ultra thick, black]  (3,3) -- (3.5,3);
\draw  [ultra thick, black]  (3.8,2) -- (4.5,2);
\draw  [ultra thick, black]  (3.8,3) -- (4.5,3);
\draw[ultra thick, black] (4.5,3) arc (90:-90:0.5);
\draw  [ultra thick, black]  (3.8,4) -- (6.5,4);
\draw  [ultra thick, black]  (3.8,1) -- (6.5,1);
%%%
\draw[ultra thick, black] (6,3) arc (90:270:0.5);
\draw  [ultra thick, black]  (6,2) -- (6.5,2);
\draw  [ultra thick, black]  (6,3) -- (6.5,3);
\draw[ultra thick, black] (7.5,3) arc (90:-90:0.5);
\draw  [ultra thick, black]  (6.8,2) -- (7.5,2);
\draw  [ultra thick, black]  (6.8,3) -- (7.5,3);
\draw  [ultra thick, black]  (6.8,4) -- (7.5,4);
\draw  [ultra thick, black]  (6.8,1) -- (7.5,1);
\draw[ultra thick, red]  (7,2.1)  arc (50:320: 0.2cm and 0.9cm);
\draw  [ultra thick, red]  (7,1.2) -- (7,1.8);
\draw[ultra thick, red]  (7,4.1)  arc (50:320: 0.2cm and 0.9cm);
\draw  [ultra thick, red]  (7,3.2) -- (7,3.8);
\end{tikzpicture}
}
\end{center}
\caption{Link $L_n$ with $3n+2$ components, case $n=3$.} \label{fig1}
\end{figure}
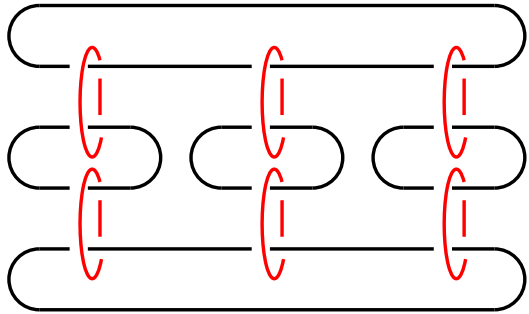
By using the
computer program 
SnapPy~\cite{Snap}, 
one can see that
the
 8-component link $L_2$ is hyperbolic and has $\operatorname{vol} (S^3 \setminus L_2) = 24.092184,$ up to six digits. To demonstrate hyperbolicity of $L_n$ for arbitrary $n \geq 2$ and find the volume formula for $S^3 \setminus L_n,$  we shall use the approach from
Lackenby~\cite{La04}
 (see 
also
Futer and Purcell~\cite{FP07}). 
In the terminology of Lackenby~\cite{La04},
 the link $L_n$ is 
  augmented,  with $2n$ "vertical"  components (colored in red in Figure~\ref{fig1}), and $S^3 \setminus L_n$ admits a decomposition into two ideal polyhedra $P_n$ and $P'_n$ with faces identified in pairs. The polyhedra $P_n$ and $P'_n$ are identical. By the construction  each vertical component of $L_n$ gives to a pair of triangles in $P_n$ with a common vertex like a bowtie, as presented in Figure~\ref{fig2} where common vertices are red.

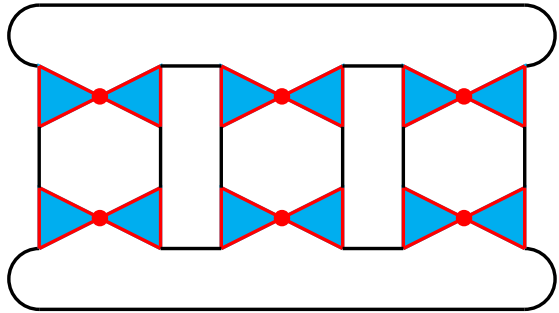
\begin{figure}[h]
\begin{center} 
\scalebox{0.8}{
\begin{tikzpicture} 
%\draw[step=5.mm, gray, very thin] (0,0) grid (100.mm,50.mm);
\draw  [ultra thick, black]  (0,0) -- (8,0);
\draw  [ultra thick, black]  (0,5) -- (8,5);
\draw[ultra thick, black] (0,1) arc (90:270:0.5);
\draw[ultra thick, black] (0,5) arc (90:270:0.5);
\draw[ultra thick, black] (8,1) arc (90:-90:0.5);
\draw[ultra thick, black] (8,5) arc (90:-90:0.5);
\filldraw  [ultra thick,  cyan]  (0,1) -- (0,2) -- (1,1.5) -- (0,1);
\filldraw  [ultra thick,  cyan]  (2,1) -- (2,2) -- (1,1.5) -- (2,1);
\filldraw  [ultra thick, cyan]  (0,3) -- (0,4) -- (1,3.5) -- (0,3);
\filldraw  [ultra thick, cyan]  (2,3) -- (2,4) -- (1,3.5) -- (2,3);
\filldraw  [ultra thick, cyan]  (3,1) -- (3,2) -- (4,1.5) -- (3,1);
\filldraw  [ultra thick, cyan]  (5,1) -- (5,2) -- (4,1.5) -- (5,1);
\filldraw  [ultra thick, cyan]  (3,3) -- (3,4) -- (4,3.5) -- (3,3);
\filldraw  [ultra thick, cyan]  (5,3) -- (5,4) -- (4,3.5) -- (5,3);
\filldraw  [ultra thick, cyan]  (6,1) -- (6,2) -- (7,1.5) -- (6,1);
\filldraw  [ultra thick, cyan]  (8,1) -- (8,2) -- (7,1.5) -- (8,1);
\filldraw  [ultra thick, cyan]  (6,3) -- (6,4) -- (7,3.5) -- (6,3);
\filldraw  [ultra thick, cyan]  (8,3) -- (8,4) -- (7,3.5) -- (8,3);
 %%%   
\draw  [ultra thick, red]  (0,1) -- (1,1.5);
\draw  [ultra thick, red]  (2,1) -- (1,1.5);
\draw  [ultra thick, red]  (0,1) -- (0,2);
\draw  [ultra thick, red]  (2,1) -- (2,2);
\draw  [ultra thick, red]  (0,2) -- (1,1.5);
\draw  [ultra thick, red]  (2,2) -- (1,1.5);
\draw  [ultra thick, red]  (0,3) -- (1,3.5);
\draw  [ultra thick, red]  (2,3) -- (1,3.5);
\draw  [ultra thick, red]  (0,3) -- (0,4);
\draw  [ultra thick, red]  (2,3) -- (2,4);
\draw  [ultra thick, red]  (0,4) -- (1,3.5);
\draw  [ultra thick, red]  (2,4) -- (1,3.5);
%%%
\draw  [ultra thick, red]  (3,1) -- (4,1.5);
\draw  [ultra thick, red]  (5,1) -- (4,1.5);
\draw  [ultra thick, red]  (3,1) -- (3,2);
\draw  [ultra thick, red]  (5,1) -- (5,2);
\draw  [ultra thick, red]  (3,2) -- (4,1.5);
\draw  [ultra thick, red]  (5,2) -- (4,1.5);
\draw  [ultra thick, red]  (3,3) -- (4,3.5);
\draw  [ultra thick, red]  (5,3) -- (4,3.5);
\draw  [ultra thick, red]  (3,3) -- (3,4);
\draw  [ultra thick, red]  (5,3) -- (5,4);
\draw  [ultra thick, red]  (3,4) -- (4,3.5);
\draw  [ultra thick, red]  (5,4) -- (4,3.5);
%%%
\draw  [ultra thick, red]  (6,1) -- (7,1.5);
\draw  [ultra thick, red]  (8,1) -- (7,1.5);
\draw  [ultra thick, red]  (6,1) -- (6,2);
\draw  [ultra thick, red]  (8,1) -- (8,2);
\draw  [ultra thick, red]  (6,2) -- (7,1.5);
\draw  [ultra thick, red]  (8,2) -- (7,1.5);
\draw  [ultra thick, red]  (6,3) -- (7,3.5);
\draw  [ultra thick, red]  (8,3) -- (7,3.5);
\draw  [ultra thick, red]  (6,3) -- (6,4);
\draw  [ultra thick, red]  (8,3) -- (8,4);
\draw  [ultra thick, red]  (6,4) -- (7,3.5);
\draw  [ultra thick, red]  (8,4) -- (7,3.5);
\draw  [ultra thick, black]  (2,1) -- (3,1);
\draw  [ultra thick, black]  (5,1) -- (6,1);
\draw  [ultra thick, black]  (2,4) -- (3,4);
\draw  [ultra thick, black]  (5,4) -- (6,4);
\draw  [ultra thick, black]  (0,2) -- (0,3);
\draw  [ultra thick, black]  (2,2) -- (2,3);
\draw  [ultra thick, black]  (3,2) -- (3,3);
\draw  [ultra thick, black]  (5,2) -- (5,3);
\draw  [ultra thick, black]  (6,2) -- (6,3);
\draw  [ultra thick, black]  (8,2) -- (8,3);
\filldraw [line width=2pt, red] (1,1.5) circle[radius=0.1cm];
\filldraw [line width=2pt, red] (4,1.5) circle[radius=0.1cm];
\filldraw [line width=2pt, red] (7,1.5) circle[radius=0.1cm];
\filldraw [line width=2pt, red] (1,3.5) circle[radius=0.1cm];
\filldraw [line width=2pt, red] (4,3.5) circle[radius=0.1cm];
\filldraw [line width=2pt, red] (7,3.5) circle[radius=0.1cm];
\end{tikzpicture}
}
\end{center}
\caption{Associating bowties to vertical components of $L_n$, case $n=3$.} \label{fig2}
\end{figure}
 
\begin{figure}[h]
\begin{center} 
\scalebox{0.8}{
\begin{tikzpicture} 
%\draw[step=5.mm, gray, very thin] (0,0) grid (100.mm,50.mm);
\draw  [ultra thick, black]  (-0.5,0) -- (8.5,0);
\draw  [ultra thick, black]  (-0.5,5) -- (8.5,5);
\draw[ultra thick, black] (-0.5,1) arc (90:270:0.5);
\draw[ultra thick, black] (-0.5,5) arc (90:270:0.5);
\draw[ultra thick, black] (8.5,1) arc (90:-90:0.5);
\draw[ultra thick, black] (8.5,5) arc (90:-90:0.5);
 %%%   red
\filldraw  [ultra thick, cyan]  (-0.5,1) -- (1,1) -- (0,2.5) -- (-0.5,1);
\filldraw  [ultra thick, cyan]  (-0.5,4) -- (1,4) -- (0,2.5) -- (-0.5,4);
\filldraw  [ultra thick, cyan]  (1,1) -- (2.5,1) -- (2,2.5) -- (1,1);
\filldraw  [ultra thick, cyan]  (1,4) -- (2.5,4) -- (2,2.5) -- (1,4);
\filldraw  [ultra thick, cyan]  (2.5,1) -- (4,1) -- (3,2.5) -- (2.5,1);
\filldraw  [ultra thick, cyan]  (2.5,4) -- (4,4) -- (3,2.5) -- (2.5,4);
\filldraw  [ultra thick, cyan]  (4,1) -- (5.5,1) -- (5,2.5) -- (4,1);
\filldraw  [ultra thick, cyan]  (4,4) -- (5.5,4) -- (5,2.5) -- (4,4);
\filldraw  [ultra thick, cyan]  (5.5,1) -- (7,1) -- (6,2.5) -- (5.5,1);
\filldraw  [ultra thick, cyan]  (5.5,4) -- (7,4) -- (6,2.5) -- (5.5,4);
\filldraw  [ultra thick, cyan]  (7,1) -- (8.5,1) -- (8,2.5) -- (7,1);
\filldraw  [ultra thick, cyan]  (7,4) -- (8.5,4) -- (8,2.5) -- (7,4);
\draw  [ultra thick, red]  (-0.5,1) -- (1,1);
\draw  [ultra thick, red]  (2.5,1) -- (1,1);
\draw  [ultra thick, red]  (-0.50,1) -- (0,2.5);
\draw  [ultra thick, red]  (2.5,1) -- (2,2.5);
\draw  [ultra thick, red]  (0,2.5) -- (1,1);
\draw  [ultra thick, red]  (2,2.5) -- (1,1);
\draw  [ultra thick, red]  (0,2.5) -- (1,4);
\draw  [ultra thick, red]  (2,2.5) -- (1,4);
\draw  [ultra thick, red]  (0,2.5) -- (-0.5,4);
\draw  [ultra thick, red]  (2,2.5) -- (2.5,4);
\draw  [ultra thick, red]  (-0.5,4) -- (1,4);
\draw  [ultra thick, red]  (2.5,4) -- (1,4);
%%%
\draw  [ultra thick, red]  (2.5,1) -- (4,1);
\draw  [ultra thick, red]  (5.5,1) -- (4,1);
\draw  [ultra thick, red]  (2.5,1) -- (3,2.5);
\draw  [ultra thick, red]  (5.5,1) -- (5,2.5);
\draw  [ultra thick, red]  (3,2.5) -- (4,1);
\draw  [ultra thick, red]  (5,2.5) -- (4,1);
\draw  [ultra thick, red]  (3,2.5) -- (4,4);
\draw  [ultra thick, red]  (5,2.5) -- (4,4);
\draw  [ultra thick, red]  (3,2.5) -- (2.5,4);
\draw  [ultra thick, red]  (5,2.5) -- (5.5,4);
\draw  [ultra thick, red]  (2.5,4) -- (4,4);
\draw  [ultra thick, red]  (5.5,4) -- (4,4);
%%%
\draw  [ultra thick, red]  (5.5,1) -- (7,1);
\draw  [ultra thick, red]  (8.5,1) -- (7,1);
\draw  [ultra thick, red]  (5.5,1) -- (6,2.5);
\draw  [ultra thick, red]  (8.5,1) -- (8,2.5);
\draw  [ultra thick, red]  (6,2.5) -- (7,1);
\draw  [ultra thick, red]  (8,2.5) -- (7,1);
\draw  [ultra thick, red]  (6,2.5) -- (7,4);
\draw  [ultra thick, red]  (8,2.5) -- (7,4);
\draw  [ultra thick, red]  (6,2.5) -- (5.5,4);
\draw  [ultra thick, red]  (8,2.5) -- (8.5,4);
\draw  [ultra thick, red]  (5.5,4) -- (7,4);
\draw  [ultra thick, red]  (8.5,4) -- (7,4);
\filldraw [line width=2pt, red] (1,1) circle[radius=0.1cm];
\filldraw [line width=2pt, red] (4,1) circle[radius=0.1cm];
\filldraw [line width=2pt, red] (7,1) circle[radius=0.1cm];
\filldraw [line width=2pt, red] (1,4) circle[radius=0.1cm];
\filldraw [line width=2pt, red] (4,4) circle[radius=0.1cm];
\filldraw [line width=2pt, red] (7,4) circle[radius=0.1cm];
\filldraw [line width=2pt, black] (0,2.5) circle[radius=0.1cm];
\filldraw [line width=2pt, black] (2,2.5) circle[radius=0.1cm];
\filldraw [line width=2pt, black] (3,2.5) circle[radius=0.1cm];
\filldraw [line width=2pt, black] (5,2.5) circle[radius=0.1cm];
\filldraw [line width=2pt, black] (6,2.5) circle[radius=0.1cm];
\filldraw [line width=2pt, black] (8,2.5) circle[radius=0.1cm];
\filldraw [line width=2pt, black] (-0.5,1) circle[radius=0.1cm];
\filldraw [line width=2pt, black] (2.5,1) circle[radius=0.1cm];
\filldraw [line width=2pt, black] (5.5,1) circle[radius=0.1cm];
\filldraw [line width=2pt, black] (8.5,1) circle[radius=0.1cm];
\filldraw [line width=2pt, black] (-0.5,4) circle[radius=0.1cm];
\filldraw [line width=2pt, black] (2.5,4) circle[radius=0.1cm];
\filldraw [line width=2pt, black] (5.5,4) circle[radius=0.1cm];
\filldraw [line width=2pt, black] (8.5,4) circle[radius=0.1cm];
\node  at (-0.8,3.8) {$A$};
\node  at (8.8,3.8) {$B$};
\node  at (-0.8,1.2) {$C$};
\node  at (8.8,1.2) {$D$};
\end{tikzpicture}
}
\end{center}
\caption{1-skeleton of the polyhedron $P_n$, case $n=3$.} \label{fig3}
\end{figure}
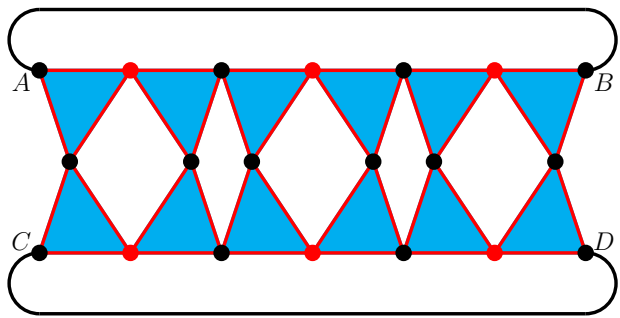

At
 the last step, in order  to obtain the polyhedron $P_n,$ we shall compress  black edges which connect two vertices of valence three, to obtain a new vertex of valence four. There are $4n$ such edges which give us $2n$ black vertices, as presented in Figure~\ref{fig3} (to complete the construction, it is necessary to identify the vertices $A$ and $B,$ as well as the vertices $C$ and $D$).  As a result, $P_n$ is an ideal right-angled polyhedron with $6n$ vertices, where $2n$ red vertices  correspond to bowties and $4n$ black vertices appeared after compressing  black edges.   Moreover, $P_n$ has two $2n$-gonal faces as its top and bottom, $2n$ triangles incident to the top, $2n$ triangles incident to the bottom, and $4n$ quadrilaterals on the middle level.  By cutting $P_n$ along the middle line passing through quadrilaterals, we shall see that $P_n$ can be decomposed into two identical ideal right-angled $2n$-gonal antiprisms $A_{2n}$. Recall that $A_n$ has $2n+2$ faces, where two $n$-gonal faces can be considered as the top and the bottom, and $2n$ triangular faces on the lateral surface, see Figure~\ref{fig4} for the antprism~$A_4$.   
 
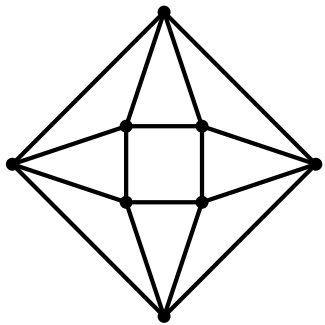
\begin{figure}[ht]
\begin{center}
\unitlength=.1mm
%1
\begin{tikzpicture}[scale=0.5] 
%\unitlength=10.mm
%\draw[step=5.mm, gray, very thin] (-15,0) grid (150.mm,100.mm);
\filldraw [line width=2pt, black] (4,5) circle[radius=0.1cm];
\filldraw [line width=2pt, black] (12,5) circle[radius=0.1cm];
\filldraw [line width=2pt, black] (8,1) circle[radius=0.1cm];
\filldraw [line width=2pt, black] (8,9) circle[radius=0.1cm];
\draw [ultra thick, black] (4,5)-- (8,9) -- (12,5) -- (8,1) -- (4,5);
\filldraw [line width=2pt, black] (7,4) circle[radius=0.1cm];
\filldraw [line width=2pt, black] (9,4) circle[radius=0.1cm];
\filldraw [line width=2pt, black] (7,6) circle[radius=0.1cm];
\filldraw [line width=2pt, black] (9,6) circle[radius=0.1cm];
\draw [ultra thick, black] (7,4)-- (7,6) -- (9,6) -- (9,4) -- (7,4);
\draw [ultra thick, black] (7,4)-- (4,5) -- (7,6) -- (8,9) --(9,6) -- (12,5) -- (9,4) -- (8,1) -- (7,4);
\end{tikzpicture}
\end{center}
\caption{1-skeleton of antiprism $A_4$.} \label{fig4}
\end{figure}

Volumes of right-angled antiprisms
 $A_n$, $n \geq 3,$
 are given by the following formula obtained by 
Thurston~\cite[Example~6.8.7]{Th80}, 
where an antiprism was named a \emph{drum with triangles}, 
 \begin{equation}
\operatorname{vol} (A_n) = 2n \left[ \Lambda \left( \frac{\pi}{4} + \frac{\pi}{2n} \right) + \Lambda \left( \frac{\pi}{4} - \frac{\pi}{2n} \right) \right].  \label{eq1}
\end{equation}
Here,
 $\Lambda(\theta)$ is the Lobachevsky function\index{Lobachevsky function} defined in~\cite[Chapter~7]{Th80} as 
$$
\Lambda (\theta) = - \int_0^{\theta} \log | 2 \sin (t) | dt.
$$
In particular,  up to six digits, we have the following volumes of ideal right-angled antiprisms: 
$$\begin{array}{l}
\operatorname{vol}(A_3) = 3.663863, \cr \operatorname{vol}(A_4)  = 6.023046, \cr \operatorname{vol}(A_5) = 8.137885, \cr \operatorname{vol}(A_6) = 10.149416.
\end{array}
$$ 

\begin{theorem} \label{th1}
For every $n \geq 2,$ the following formula holds:
$$
\operatorname{vol} (S^3 \setminus L_n) = 16 n   \left[ \Lambda \left( \frac{\pi}{4} + \frac{\pi}{4n} \right) + \Lambda \left( \frac{\pi}{4} - \frac{\pi}{4n} \right) \right].
$$
\end{theorem}

\begin{proof}
Indeed, from the above considerations we obtain
$$
\operatorname{vol} (S^3 \setminus L_n) = 2 \operatorname{vol} (P_n) = 4 \operatorname{vol} (A_{2n}),
$$
so the assertion
 follows by formula~(\ref{eq1}). 
\end{proof}

The following  was proved by 
Adams~\cite[Corollary~5.1]{Ad85}. 
\begin{theorem}[see
 \cite{Ad85}] \label{th2}
Let $J$ be a link in $S^3$ such that $S^3 \setminus J$ is hyperbolic and $J$ has a projection for which some part appears as in Figure~\ref{fig5}(a). 
Let $J'$ be the link obtained by replacing that part by the projection of $J$ appearing in Figure~\ref{fig5}(a)  with
 the one
  appearing in Figure~\ref{fig5}(b). Then $S^3 \setminus J'$ is hyperbolic with the same volume as $S^3 \setminus J$. 
\end{theorem}

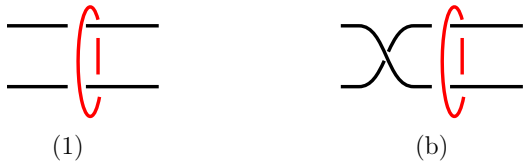
\begin{figure}[h]
\begin{center} 
\scalebox{0.8}{
\begin{tikzpicture} 
%\draw[step=5.mm, gray, very thin] (0,0) grid (150.mm,50.mm);
 % 1 %%   
 \draw  [ultra thick, black]  (0.5,2) -- (1.5,2);
\draw  [ultra thick, black]  (0.5,1) -- (1.5,1);
\draw  [ultra thick, black]  (1.8,2) -- (3,2);
\draw  [ultra thick, black]  (1.8,1) -- (3,1);
\draw[ultra thick, red]  (2,2.1)  arc (50:320: 0.2cm and 0.9cm);
\draw  [ultra thick, red]  (2,1.2) -- (2,1.8);
\node  at (1.5,0) {(1)};
% 2 %%
\pic[
 rotate=90,
 braid/.cd,
 every strand/.style={ultra thick},
 strand 1/.style={black},
 strand 2/.style={black}] 
at (6,1) {braid={s_1^{-1} }}; 
 %%%   
\draw  [ultra thick, black]  (7.8,2) -- (9,2);
\draw  [ultra thick, black]  (7.8,1) -- (9,1);
\draw[ultra thick, red]  (8,2.1)  arc (50:320: 0.2cm and 0.9cm);
\draw  [ultra thick, red]  (8,1.2) -- (8,1.8);
\node  at (7.5,0) {(b)};
%%%
 \end{tikzpicture}
}
\end{center}
\caption{Replacing the part of the projection by Theorem~\ref{th2}.} \label{fig5}
\end{figure}

We shall call the replacement presented in Figure~\ref{fig5} and its inverse
the  \emph{Adams moves}. 
Let us apply Adams moves to each of $2n$ parts of related  vertical components of the hyperbolic link $L_n$. The resulting link $L'_n$ with $3n$ components is 
depicted
 in Figure~\ref{fig6} for the case $n=3$. 

\begin{figure}[h]
\begin{center} 
\scalebox{0.8}{
\begin{tikzpicture} 
%\draw[step=5.mm, gray, very thin] (0,0) grid (150.mm,50.mm);
\draw  [ultra thick, black]  (0,0) -- (10.5,0);
\draw  [ultra thick, black]  (0,5) -- (10.5,5);
\draw[ultra thick, black] (0,1) arc (90:270:0.5);
\draw[ultra thick, black] (0,5) arc (90:270:0.5);
\draw[ultra thick, black] (10.5,1) arc (90:-90:0.5);
\draw[ultra thick, black] (10.5,5) arc (90:-90:0.5);
 \pic[
 rotate=90,
 braid/.cd,
 every strand/.style={ultra thick},
 strand 1/.style={black},
 strand 2/.style={blue}] 
at (0,1) {braid={s_1 }}; 
 \pic[
 rotate=90,
 braid/.cd,
 every strand/.style={ultra thick},
 strand 1/.style={blue},
 strand 2/.style={black}] 
at (0,3) {braid={s_1^{-1} }}; 
 %%%   
\draw[ultra thick, blue] (0,3) arc (90:270:0.5);
\draw  [ultra thick, black]  (1.8,2) -- (2.5,2);
\draw  [ultra thick, black]  (1.8,3) -- (2.5,3);
\draw[ultra thick, black] (2.5,3) arc (90:-90:0.5);
\draw  [ultra thick, blue]  (1.8,4) -- (4,4);
\draw  [ultra thick, blue]  (1.8,1) -- (4,1);
\draw[ultra thick, red]  (2,2.1)  arc (50:320: 0.2cm and 0.9cm);
\draw  [ultra thick, red]  (2,1.2) -- (2,1.8);
\draw[ultra thick, red]  (2,4.1)  arc (50:320: 0.2cm and 0.9cm);
\draw  [ultra thick, red]  (2,3.2) -- (2,3.8);
%%%
 \pic[
 rotate=90,
 braid/.cd,
 every strand/.style={ultra thick},
 strand 1/.style={blue},
 strand 2/.style={cyan}] 
at (4,1) {braid={s_1 }}; 
 \pic[
 rotate=90,
 braid/.cd,
 every strand/.style={ultra thick},
 strand 1/.style={cyan},
 strand 2/.style={blue}] 
at (4,3) {braid={s_1^{-1}}}; 
\draw[ultra thick, red]  (6,2.1)  arc (50:320: 0.2cm and 0.9cm);
\draw  [ultra thick, red]  (6,1.2) -- (6,1.8);
\draw[ultra thick, red]  (6,4.1)  arc (50:320: 0.2cm and 0.9cm);
\draw  [ultra thick, red]  (6,3.2) -- (6,3.8);
\draw[ultra thick, cyan] (4,3) arc (90:270:0.5);
\draw  [ultra thick, blue]  (5.8,2) -- (6.5,2);
\draw  [ultra thick, blue]  (5.8,3) -- (6.5,3);
\draw[ultra thick, blue] (6.5,3) arc (90:-90:0.5);
\draw  [ultra thick, cyan]  (5.8,4) -- (8,4);
\draw  [ultra thick, cyan]  (5.8,1) -- (8,1);
%%%
 \pic[
 rotate=90,
 braid/.cd,
 every strand/.style={ultra thick},
 strand 1/.style={black},
 strand 2/.style={cyan}] 
at (8,3) {braid={s_1^{-1} }}; 
%%%
 \pic[
 rotate=90,
 braid/.cd,
 every strand/.style={ultra thick},
 strand 1/.style={cyan},
 strand 2/.style={black}] 
at (8,1) {braid={s_1 }}; 
\draw[ultra thick, black] (8,3) arc (90:270:0.5);
\draw[ultra thick, cyan] (10.5,3) arc (90:-90:0.5);
\draw  [ultra thick, cyan]  (9.8,2) -- (10.5,2);
\draw  [ultra thick, cyan]  (9.8,3) -- (10.5,3);
\draw  [ultra thick, black]  (9.8,4) -- (10.5,4);
\draw  [ultra thick, black]  (9.8,1) -- (10.5,1);
\draw[ultra thick, red]  (10,2.1)  arc (50:320: 0.2cm and 0.9cm);
\draw  [ultra thick, red]  (10,1.2) -- (10,1.8);
\draw[ultra thick, red]  (10,4.1)  arc (50:320: 0.2cm and 0.9cm);
\draw  [ultra thick, red]  (10,3.2) -- (10,3.8);
\end{tikzpicture}
}
\end{center}
\caption{Link $L'_n$ with $3n$ components, case $n=3$.} \label{fig6}
\end{figure}
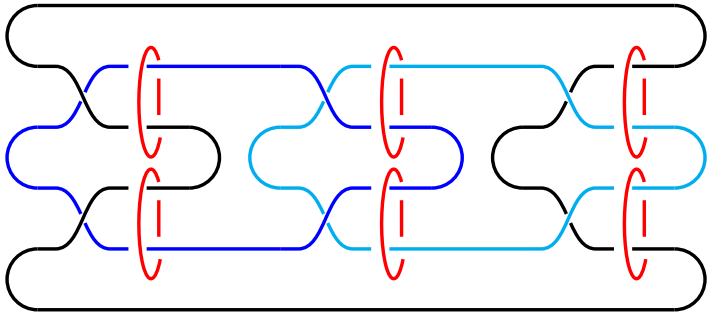

Theorems~\ref{th1} and~\ref{th2} now yield the following result. 
\begin{corollary} \label{cor}
For every $n \geq 2,$ the following formula holds:
\begin{equation}
\operatorname{vol} (S^3 \setminus L'_n) = 16 n   \left[ \Lambda \left( \frac{\pi}{4} + \frac{\pi}{4n} \right) + \Lambda \left( \frac{\pi}{4} - \frac{\pi}{4n} \right) \right]. 
\end{equation}
\end{corollary}

\section{A family of hyperbolic Brunnian links} \label{sec3} 

Suppose that $\overline{M}$ is a compact orientable 3-manifold with $\partial M$ a collection of tori, and that the interior $M \subset \overline{M}$ admits a complete hyperbolic structure. Then $M$ is referred to as a \emph{cusped manifold}. For any $n\geq 2,$ the complement $S^3 \setminus L'_n$ is a cusped manifold by construction. 

For every $n\geq 2$ and every positive integers $k_1, \ldots,  k_n,$ we define an $n$-component link $Br(k_1, \ldots, k_n)$ with a diagram having $2n$ blocks of twist regions as follows: Consider $n$ top blocks consisting of  $2k_i +1$ positive half-twists and $n$  bottom blocks  consisting of $2k_i + 1$ negative half-twists for $i=1, \ldots, n$.  A diagram of $Br(1,2,1)$ is presented in~Figure~\ref{fig7}. 
By
Rolfsen~\cite{Ro76}, 
 $Br(k_1, \ldots, k_n)$ can be considered as a result of Dehn fillings with slopes $1/k_i$ and $-1/k_i$ for $i=1, \ldots, n$ on $2n$ cusps corresponding to vertical components of $S^3 \setminus L'_n$ in Figure~\ref{fig6}.  Thus, in total we get $2n$ twist regions with $2k_i+1$ positive and negative half-twists, $i=1, \ldots, n$.   
%%%
\begin{figure}[h]
\begin{center} 
\scalebox{0.8}{
\begin{tikzpicture} 
%\draw[step=5.mm, gray, very thin] (0,0) grid (150.mm,50.mm);
\draw  [ultra thick, black]  (0,0) -- (15.5,0);
\draw  [ultra thick, black]  (0,5) -- (15.5,5);
\draw[ultra thick, black] (0,1) arc (90:270:0.5);
\draw[ultra thick, black] (0,5) arc (90:270:0.5);
\draw[ultra thick, black] (15.5,1) arc (90:-90:0.5);
\draw[ultra thick, black] (15.5,5) arc (90:-90:0.5);
 \pic[
 rotate=90,
 braid/.cd,
 every strand/.style={ultra thick},
 strand 1/.style={black},
 strand 2/.style={blue}] 
at (0,1) {braid={s_1 s_1 s_1}}; 
 \pic[
 rotate=90,
 braid/.cd,
 every strand/.style={ultra thick},
 strand 1/.style={blue},
 strand 2/.style={black}] 
at (0,3) {braid={s_1^{-1} s_1^{-1} s_1^{-1}}}; 
\node  at (1.75,0.5) {$k_1=1$};
\node  at (1.75,4.5) {$k_1=1$};
 %%%   
\draw[ultra thick, blue] (0,3) arc (90:270:0.5);
\draw[ultra thick, black] (3.5,3) arc (90:-90:0.5);
\draw  [ultra thick, blue]  (3.5,4) -- (5,4);
\draw  [ultra thick, blue]  (3.5,1) -- (5,1);
%%%
 \pic[
 rotate=90,
 braid/.cd,
 every strand/.style={ultra thick},
 strand 1/.style={blue},
 strand 2/.style={cyan}] 
at (5,1) {braid={s_1 s_1 s_1 s_1 s_1}}; 
 \pic[
 rotate=90,
 braid/.cd,
 every strand/.style={ultra thick},
 strand 1/.style={cyan},
 strand 2/.style={blue}] 
at (5,3) {braid={s_1^{-1} s_1^{-1} s_1^{-1}  s_1^{-1}  s_1^{-1}}}; 
\draw[ultra thick, cyan] (5,3) arc (90:270:0.5);
\draw[ultra thick, blue] (10.5,3) arc (90:-90:0.5);
\draw  [ultra thick, cyan]  (10.5,4) -- (12,4);
\draw  [ultra thick, cyan]  (10.5,1) -- (12,1);
\node  at (7.75,0.5) {$k_2=2$};
\node  at (7.75,4.5) {$k_2=2$};
%%%
 \pic[
 rotate=90,
 braid/.cd,
 every strand/.style={ultra thick},
 strand 1/.style={black},
 strand 2/.style={cyan}] 
at (12,3) {braid={s_1^{-1} s_1^{-1} s_1^{-1}}}; 
%%%
 \pic[
 rotate=90,
 braid/.cd,
 every strand/.style={ultra thick},
 strand 1/.style={cyan},
 strand 2/.style={black}] 
at (12,1) {braid={s_1 s_1 s_1}}; 
\draw[ultra thick, black] (12,3) arc (90:270:0.5);
\draw[ultra thick, cyan] (15.5,3) arc (90:-90:0.5);
\node  at (13.75,0.5) {$k_3=1$};
\node  at (13.75,4.5) {$k_3=1$};
\end{tikzpicture}
}
\end{center}
\caption{Link $Br(1,2,1)$.} \label{fig7}
\end{figure}
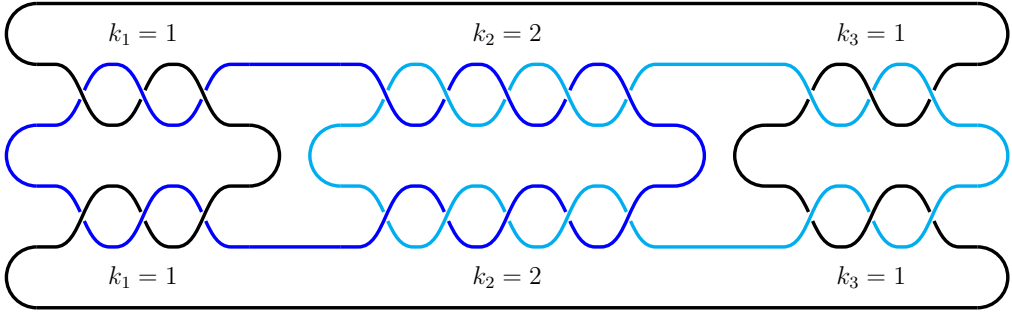

If all $k_i$ are equal to some $k$, then $Br(k, k, \ldots, k)$ admits a cyclic rotation symmetry of order $n$.  The links $Br(1,1, \ldots , 1)$ appeared
already 
 in Debrunner~\cite[Fig.~2]{De61} and were denoted by $L_F$. It was demonstrated
 in~\cite{De61} that $L_F$ is unsplittable, but any proper sublink of $L_F$ is completely splittable. The 5-component link presented in Rolfsen~\cite[p.~69]{Ro76} 
  is link $Br(1, 1, 1, 1, 1)$ in our notation. By the same argument as
   in~\cite{De61}, the following result holds. 

\begin{theorem} \label{th3}
For every $n,$ $Br(k_1,  \ldots, k_n)$ is an $n$-component Brunnian link. 
\end{theorem}

\begin{proof}
It is easy to see that after removing one component of $Br(k_1,  \ldots, k_n),$  all other components can be transformed to trivial ones by a sequence of simplifying underpassing moves from (a) to (b) shown in Figure~\ref{fig8}. 
%%%
\begin{figure}[h]
\begin{center} 
\scalebox{0.8}{
\begin{tikzpicture} 
%\draw[step=5.mm, gray, very thin] (0,0) grid (150.mm,50.mm);
 \pic[
 rotate=90,
 braid/.cd,
 every strand/.style={ultra thick},
 strand 1/.style={black},
 strand 2/.style={blue}] 
at (0,1) {braid={s_1 }}; 
 \pic[
 rotate=90,
 braid/.cd,
 every strand/.style={ultra thick},
 strand 1/.style={blue},
 strand 2/.style={black}] 
at (0,3) {braid={s_1^{-1} }}; 
\draw[ultra thick, black] (1.5,3) arc (90:-90:0.5);
 %%%   
\draw[ultra thick, blue] (1.5,4) arc (90:-90:1.5);
\node  at (1,0) {(a)};
\draw  [ultra thick, blue]  (5,3) -- (5.5,3);
\draw  [ultra thick, blue]  (5,2) -- (5.5,2);
\draw[ultra thick, blue] (5.5,3) arc (90:-90:0.5);
\draw  [ultra thick, black]  (5,4) -- (5.5,4);
\draw  [ultra thick, black]  (5,1) -- (5.5,1);
\draw  [ultra thick, black]  (5.5,4) -- (6.5,3);
\draw  [ultra thick, black]  (6.5,3) -- (7,3);
\draw  [ultra thick, black]  (5.5,1) -- (6.5,2);
\draw  [ultra thick, black]  (6.5,2) -- (7,2);
\draw[ultra thick, black] (7,3) arc (90:-90:0.5);
\node  at (6,0) {(b)};
\end{tikzpicture}
}
\end{center}
\caption{The simplifying underpassing move.} \label{fig8}
\end{figure}
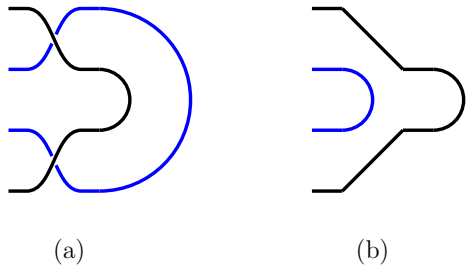
\end{proof}

A practical method to check the hyperbolicity of Brunnian links was presented by
Bai~\cite{Ba21}
who proved
that for $n\geq 2$ with $k_i=1$, $i=1, \ldots, n$, the link $Br(k_1, \ldots, k_n)$  is hyperbolic~\cite[Theorem~1.3]{Ba21}. The smallest one is a 2-component link $Br(1,1)$ with 12 crossings and $\operatorname{vol} (S^3 \setminus Br(1,1)) = 12. 528922,$ up to six digits. Moreover, it can be recognized by using SnapPy~\cite{Snap} that  $Br(1,1) = L12n1180$, a non-alternating link with 12 components. At the same time, for $k_i \geq 3,$ twisted regions in the diagram of $Br(k_1, \ldots, k_n)$ have more than $6$ crossings. Therefore, the following result from Futer and Purcell~\cite[Theorem~1.7]{FP07} can be applied.

\begin{theorem}[see \cite{FP07}]
Let $K$ be a link in $S^3$ with a prime, twist-reduced diagram $D(K)$. Suppose that every twist region of $D(K)$ contains at least $6$ crossings and that each component of $K$ passes through at least $7$ twist regions (counted with their multiplicity). Then every non-trivial Dehn filling of all the components of $K$ is hyperbolic. 
\end{theorem} 

The relation between volumes of cusped manifold and hyperbolic manifolds obtained by Dehn filling is due to Gromov and Thurston~\cite[Theorem~6.5.6]{Th80}.

\begin{theorem} [see \cite{Th80}] \label{th4}
Suppose $M$ is a complete hyperbolic manifold of finite volume and that $N \neq M$ is a complete hyperbolic manifold obtained topologically by replacing certain cusps of $M$ by solid tori. Then $\operatorname{vol} (N) < \operatorname{vol} (M)$. 
\end{theorem}
   
\begin{theorem} \label{th5}
For hyperbolic links $Br (k_1, \ldots, k_n),$ the following upper bound holds: 
$$
\operatorname{vol} (S^3 \setminus B(k_1, \ldots, k_n)) <  16 n   \left[ \Lambda \left( \frac{\pi}{4} + \frac{\pi}{4n} \right) + \Lambda \left( \frac{\pi}{4} - \frac{\pi}{4n} \right) \right].
$$
\end{theorem}	   

\begin{proof}
Since $Br(k_1, \ldots, k_n)$ can be obtained by Dehn filling on $2n$ cusps of the hyperbolic link $L'_n$, the result follows by
Corollary~\ref{cor}
and
 Theorem~\ref{th4}.
\end{proof}

\begin{corollary}
For every $n \geq 2,$ the value $\beta_n = 16 n  \left[ \Lambda \left( \frac{\pi}{4} + \frac{\pi}{4n} \right) + \Lambda \left( \frac{\pi}{4} - \frac{\pi}{4n} \right) \right]$ is the limit point for volumes of hyperbolic Brunnian links with $n$ components.
\end{corollary}

We  conclude the chapter by the following open problems concerning hyperbolic Brunnian links. 

\begin{problem}
What is the smallest volume hyperbolic Brunnian link with $n$ components?
\end{problem}

It is well known that the link "Borromean rings" is arithmetic~\cite[Chapter~7]{Th80}. 

\begin{problem}
Which Brunnian links are arithmetic?  
\end{problem}

Also, the following interesting problem was formulated by Bai and Ma~\cite[Problem 7.0.8]{BM21}. 

\begin{problem}[see~\cite{BM21}]
Let $B(n)$ be the number of Brunnian links with $n$ or fewer crossings, $B_h(n)$ the number of hyperbolic Brunnian links with $n$ or fewer crossings, and denote
$$
\limsup_{n \to \infty} \frac{B_h(n)}{B(n)} = a, \qquad \liminf_{n \to \infty} \frac{B_h(n)}{B(n)} = b.
$$
Is then $a=b$, $b=0,$ or $a<1$?
\end{problem}

\subsection*{Acknowledgements}
We thank the editors and referees for comments and suggestions.

\end{document}